\documentclass[11pt,a4paper]{amsart}
\usepackage{amssymb}
\usepackage[final]{showkeys}
\usepackage{microtype}
\usepackage{color}
\usepackage{graphicx}
\usepackage[hmargin=3cm,vmargin={3.5cm,4cm}]{geometry}

\newtheorem{te}{Theorem}[section]

\newtheorem{os}[te]{Remark}
\newtheorem{prop}[te]{Proposition}

\numberwithin{equation}{section}

\allowdisplaybreaks

\begin{document}

    \title[Distribution for a planar motion]{Random flights related to the Euler-Poisson-Darboux equation}

    \author{Roberto Garra$^1$}
    \address{${}^1$Dipartimento di Scienze di Base e Applicate per l'Ingegneria, ``Sapienza'' Universit\`a di Roma.}

    \author{Enzo Orsingher$^2$}
    \address{${}^2$Dipartimento di Scienze Statistiche, ``Sapienza'' Universit\`a di Roma.}


    \date{\today}

    \begin{abstract}
     This paper is devoted to the analysis of random motions on the line and in the space $\mathbb{R}^d$
     ($d>1$) performed at finite velocity and governed by a non-homogeneous
     Poisson process with rate $\lambda(t)$. The explicit distributions $p(\mathbf{x},t)$ of the position of 
     the randomly moving particles are obtained by solving initial-value problems for the Euler-Poisson-Darboux
     equation when $\lambda(t)= \frac{\alpha}{t}$, $\alpha >0$. We consider also the case where $\lambda(t)= \lambda \coth \lambda t$
     and $\lambda(t)= \lambda \tanh\lambda t$ where some Riccati differential equations emerge and the explicit distributions
     are obtained for $d=1$. We also examine planar random motions with random velocities obtained by projecting random flights in $\mathbb{R}^d$
     onto the plane. Finally the case of planar motions with four orthogonal directions is considered and the corresponding 
     higher-order equations with time-varying coefficients obtained.

    \textbf{Mathematics~Subject~Classification~(2010):} 60G60
    \end{abstract}

    \maketitle

    \section{Introduction}
    The famous Euler-Poisson-Darboux (EPD) equation 
    \begin{equation}\label{1}
    \frac{\partial^2 u}{\partial t^2}+\frac{2\alpha}{t}\frac{\partial u}{\partial t}= c^2 \bigg\{\sum_{j=1}^d\frac{\partial^2 u}{\partial x_j^2}\bigg\},
    \quad \mathbf{x}= (x_1, \dots, x_d)\in \mathbb{R}^d, t>0, \alpha>0,
    \end{equation}
    can be regarded as a multidimensional telegraph equation. For $d=1$, \eqref{1}
    can be viewed as the governing equation of a telegraph process where changes of direction
    are paced by a non-homogeneous Poisson process with rate $\lambda(t)= \frac{\alpha}{t}$ (see \cite{Foong}).
    In this paper we present in the one-dimensional case and in the Euclidean space $\mathbb{R}^d$ random motions
    at finite velocity whose distributions are obtained as solutions of Cauchy problems for the EPD equation.\\
    In order to catch the flavor
    of what is going on in this work, we consider the symmetric telegraph process $\mathcal{T}(t)$ performed at velocity $c$ and
    with reversals of the direction of motion timed by the non-homogeneous Poisson process
    of rate $\lambda(t)= \frac{\alpha}{t}$. The law of $\mathcal{T}(t)$, say $p(x,t)$, is related to the Cauchy problem
    \begin{equation}\label{1.1}
    \begin{cases}
    \frac{\partial^2 v}{\partial t^2}+\frac{2\alpha}{t}\frac{\partial v}{\partial t}= c^2 \frac{\partial^2 v}{\partial x^2},\\
    v(x,0)=f(x),\\
    \frac{\partial v}{\partial t}(x,t)\big|_{t=0}=0,
    \end{cases}
    \end{equation}
    whose solution can be represented as the Erd\'elyi-Kober fractional integral of the D'Alembert solution
    of the wave equation (see \cite{erd})
    \begin{equation}
    \begin{cases}
    \frac{\partial^2 w}{\partial t^2}= c^2 \frac{\partial^2 w}{\partial x^2},\\
    w(x,0)= f(x),\\
    \frac{\partial w}{\partial t}(x,t)\big|_{t=0}=0.
    \end{cases}
    \end{equation}
    
    This means that
    \begin{align}
    \nonumber v(x,t)&=\frac{2}{B(\alpha, \frac{1}{2})}\int_0^1 du(1-u^2)^{\alpha-1}\left[\frac{f(x+uct)+f(x-uct)}{2}\right]\\
    & =\frac{1}{B(\alpha, \frac{1}{2})\, ct}\int_{x-ct}^{x+ct}\left(1-\left(\frac{x-u}{ct}\right)^2\right)^{\alpha-1}f(u)du,\label{2}
    \end{align}
    where $B(\alpha,\frac{1}{2})$ is the Beta function of parameters $\alpha$ and $1/2$.
    For $f(x)=\delta(x)$, where $\delta(x)$, $x\in \mathbb{R}$ is the Dirac delta function,
    we extract from  \eqref{2} the distribution (see \cite{Foong})
    \begin{equation}\label{fon}
    P\{\mathcal{T}(t)\in dx\}/dx= \frac{1}{B(\alpha, \frac{1}{2})\, ct}\left(1-\frac{x^2}{c^2t^2}\right)^{\alpha-1},
    \quad |x|<ct,
    \end{equation}
    which unlike the classical telegraph process, has no discrete components. This is because initially the rate of the Poisson
    process makes the changes of direction infinitely frequent and thus the moving particle can not reach the endpoints $\pm ct$.\\
    Moreover, we notice that, for $\alpha = 1/2$, we have
        \begin{equation}
        P\{\mathcal{T}(t)\in dx\}/dx= \frac{1}{\pi \sqrt{c^2t^2-x^2}}, \quad |x|<ct,
        \end{equation}
        which coincides with the arcsine law. \\
    In section 2 we present solutions to the EPD equation (including the non-homogeneous one examined in Theorem 2.1) with different initial conditions
    and provide probabilistic representations of the results.
    
    For the $d$-dimensional EPD  
    \begin{equation}\label{epos}
        \left(\frac{\partial^2}{\partial t^2}+\frac{d+2\gamma-1}{t}\frac{\partial}{\partial
        t}\right)p(\mathbf{x},t)= c^2 \displaystyle \sum_{j=1}^{d}\frac{\partial^2}{\partial x_j^2} p(\mathbf{x},t),
    \end{equation}
    we show that the fundamental solution coincides with the following probability density
    \begin{equation}\label{epoin}
        p(\mathbf{x},t) =
        \frac{\Gamma(\gamma+\frac{d}{2})}{\pi^{d/2}\Gamma(\gamma)}\frac{1}{(ct)^{d-2+2\gamma}}\left(c^2t^2-\|\mathbf{x}\|^2\right)^{\gamma-1},
        \quad \|\mathbf{x}\|^2< c^2t^2,\; \gamma>0.
        \end{equation}
        
     For $d=1$, the distribution \eqref{epoin} coincides with \eqref{fon} and thus can be regarded as
     its multidimensional extension. Furthermore, the one-dimensional marginals of \eqref{epoin}
     have the form \eqref{fon} with $\alpha = \gamma+\frac{d-1}{2}$ and therefore \eqref{epoin} can be associated with a random flight
     whose projection on the line behaves as a telegraph process with velocity reversals
     governed by the non-homogeneous Poisson process with rate $\alpha/t$.
     
     For special values of $\alpha = \frac{n}{2}(d-1)$, the distribution \eqref{epoin} coincides
     with the conditional distributions of random flights with Dirichlet displacements treated in \cite{Ale,lec,lec1}.
     The projection of random flights on lower dimensional spaces can be described as random motions with random
     velocity. We devote section 4.1 to the analysis of planar random flights with random velocities
     (in the one-dimensional case Stadje and Zacks in \cite{staz} have investigated this type of telegraph processes).
     Furthermore for $d=2$ and $\alpha = \frac{n\gamma}{2}$, with $\gamma\in (0,1)$, the probability law
     \eqref{epoin} emerges from the fractional extension of the planar random motions with an infinite
     number of directions (see \cite{noi}). 
     
     Telegraph equations related to the one-dimensional telegraph process with a non-homogeneous
     Poisson process have been considered in the cases where $\lambda(t)= \lambda \tanh \lambda t$ by Iacus in \cite{Iacus}
     and here for $\lambda(t)= \lambda \coth\lambda t$. These two cases substantially differ
     because for $\lambda(t)=\lambda \coth\lambda t$ the random motion has only probability density 
     in $(-ct,+ct)$ while in the Iacus model, the process has a discrete component at $\pm ct$
     as in the classical telegraph process.
     In Section 3 we prove also that the law of the classical telegraph process is a suitable combination of these telegraph 
     processes with time-varying rates.
     
     Section 5 is devoted to planar random motions with four orthogonal directions of motion
     where changes of direction are governed by a non-homogeneous Poisson process with rate $\lambda(t)$.
     We are able to derive the general equation governing the distribution of the position 
     $(\mathcal{X}(t), \mathcal{Y}(t))$ which coincides with the fourth-order p.d.e. obtained in 
     \cite{koll} when $\lambda(t) = \lambda = const.$
     In this case the distribution of $(\mathcal{X}(t), \mathcal{Y}(t))$ has the following 
     representation (see \cite{Ors})
     \begin{equation}\label{zeroo}
      		    	    \begin{cases}
      		    	    &\mathcal{X}(t)= \mathcal{U}(t)+\mathcal{V}(t),\\
      		    	    & \mathcal{Y}(t)= \mathcal{U}(t)-\mathcal{V}(t),
      		    	    \end{cases}
      		    	    \end{equation}
      		    	    where $\mathcal{U}(t)$ and $\mathcal{V}(t)$ are independent telegraph processes with parameters 
      		    	    $\lambda/2$ and velocity $c/2$. In principle, we can use the scheme \eqref{zeroo}
      		    	    in constructing planar random motions by using the one-dimensional processes
      		    	    related to the EPD equation studied in this paper. This, however, leads to different governing equations.
   	
   	\section{One dimensional EPD equation and telegraph processes}

	A random motion $\mathcal{T}(t)$, $t>0$, with rightward and backward velocity equal to $\pm c$, respectively,
	where changes of direction are paced by a non-homogeneous Poisson process, denoted by $\mathcal{N}(t) $, with time-dependent
	rate $\lambda(t)$, $t>0$, has distribution $p(x,t)$ satisfying
	the Cauchy problem (see e.g. \cite{kaplan})
	\begin{equation}\label{sec2}
	\begin{cases}
	\frac{\partial^2 p}{\partial t^2}+2\lambda(t)\frac{\partial p}{\partial t}= c^2\frac{\partial^2 p}{\partial x^2},\\
	p(x,0)=\delta(x),\quad \frac{\partial p}{\partial t}(x,t)\big|_{t=0}=0.
	\end{cases}
	\end{equation}
	The joint distributions
	 \begin{equation}\label{pri}
	   p_k(x,t) \, dx=P\{\mathcal{T}(t)\in dx, \mathcal{N}(t)=
	    k\},
	    \end{equation}
	satisfy, instead, the second-order difference-differential equations with time-varying coefficients
	\begin{equation}\label{sec}
	    \frac{\partial^2p_k}{\partial t^2}+2\lambda(t)\frac{\partial p_k}{\partial
	    	    t}-c^2\frac{\partial^2 p_k}{\partial x^2}=
	    \lambda^2(t)p_{k-2}-\lambda^2(t)p_k+\lambda'(t)p_{k-1}-\lambda'(t)p_k,
	  \end{equation}
    under the initial conditions
    \begin{align}
    \nonumber &p_k(x,0)= \begin{cases}
    0, \; &k\geq 1\\
    \delta(x), \; &k=0,
    \end{cases}\\
    \nonumber &\frac{\partial p_k}{\partial t}(x,t)\big|_{t=0}=0, \quad k\geq 0,
    \end{align}
    provided that the initial velocity is symmetrically distributed.
	For $\lambda(t) = \lambda$, equation \eqref{sec}  coincides with equation Eq.(2.8) in \cite{ales}.\\
	For $\lambda(t)=\frac{\alpha}{t}$, equation \eqref{sec2} reduces to the EPD equation. 
	The physical meaning of this EPD-type equation in the framework of heat waves studies has been 
	recently discussed in \cite{bar}.
	
	It is well-known (see \cite{erd} and \cite{samko}, section 41) that the Erd\'elyi-Kober
	fractional integral of the D'Alembert solution of the wave equation is given by equation \eqref{2} and 
	solves the Cauchy problem for the EPD equation
	\begin{equation}\label{beler}
			\begin{cases}
			\left(\frac{\partial^2}{\partial t^2}+\frac{2\alpha}{t}\frac{\partial}{\partial t}\right)v(x,t)= c^2\frac{\partial^2}{\partial x^2}v(x,t),\\
			v(x,0)= f(x),\\
			\frac{\partial v}{\partial t}(x,t)\big|_{t=0}=0,
			\end{cases}
			\end{equation}
				
	 The first probabilistic derivation of this interesting relation was given by Rosencrans in \cite{Rosen}.
	Indeed, the solution \eqref{2} can be written also as
	 \begin{equation}
	    v(x,t) = \mathbb{E}\left(\frac{f(x+c\,\mathcal{U}(t))+f(x-c\,\mathcal{U}(t))}{2}\right),	
	 \end{equation}
	where
	\begin{equation}\label{agost}
	\mathcal{U}(t) = U(0)\int_0^t(-1)^{\mathcal{N}(s)}ds,
	\end{equation}
	where $\mathcal{N}(t)$ is the non-homogeneous Poisson process with rate $\lambda(t)= \frac{\alpha}{t}$.
	The r.v. $U(0)$ is independent from $\mathcal{N}(t)$ and $P(U(0)=\pm c)= \frac{1}{2}$.
	The distribution of $\mathcal{U}(t)$ can be extracted from equation \eqref{2} by assuming $f(u)= \delta(u)$, where $\delta(\cdot)$
	is the Dirac delta function, and
	reads
	\begin{align}
	 \nonumber P\{\mathcal{T}(t)\in dx\}/dx &= \frac{1}{B(\alpha, \frac{1}{2})\, ct}\left(1-\frac{x^2}{c^2t^2}\right)^{\alpha-1}\\
	 \label{fonsec} &= \frac{(ct)^{1-2\alpha}}{B(\alpha, \frac{1}{2})}\left(c^2t^2- x^2\right)^{\alpha-1},\quad |x|<ct.
	 \end{align}
	 
	 The density function \eqref{fonsec} coincides with formula (25) in \cite{Foong} and is the fundamental solution to the one-dimensional EPD equation (see \cite{Brest}). The reason for which \eqref{fonsec} integrates to one, and thus the telegraph process considered
	 here has no singular component, depends from the form of the Poisson rate which hinders the particle to reach the boundaries
	 $\pm ct$ because the direction initially changes infinitely often. The probability density \eqref{fonsec} for $0<\alpha<1$
	 has an arcsine law structure (and thus for $|x|\rightarrow \pm ct$ tends to infinite), while for $\alpha>1$
	 has a bell-shaped form. In the latter case this means that the moving particle 
	 oscillates around the origin. 
	 
	 For $\alpha =1$ formula \eqref{fonsec} corresponds to a uniform law in $(-ct,+ct)$.
	 A special case is that of $\alpha =\frac{1}{2}$ where the normalizing factor does not depend on time.
	 We finally observe that 
	 \begin{equation}
	 \mathbb{E}\mathcal{T}^{2k}(t)=(ct)^{2k+2}\frac{\Gamma(\alpha+\frac{1}{2})\Gamma(k+\frac{1}{2})}{\sqrt{\pi}\Gamma(\alpha+k+\frac{1}{2})}.
	 \end{equation}
	 
	We now arrive at the non-homogeneous EPD equation and explore its connection with the forced wave equation through
	the Erd\'elyi-Kober integrals (see formula \eqref{for1} below). We have therefore the following theorem.

	\begin{te}
	The solution of the Cauchy problem for the forced wave equation
	\begin{equation}\label{for0}
		    \begin{cases}
		    \frac{\partial^2 w}{\partial t^2}= c^2 \frac{\partial^2 w}{\partial x^2}+F(x,t), \quad F\in C^2[\mathbb{R}\times[0,\infty)],\\
		    w(x,0)= 0,\\
		    \frac{\partial w}{\partial t}(x,t)\big|_{t=0}=0,
		    \end{cases}
	    \end{equation}
	is related to the solution of the non-homogeneous EPD equation
	\begin{equation}\label{for}
	    \begin{cases}
	    \frac{\partial^2 v}{\partial t^2}+\frac{2\alpha}{t}\frac{\partial v}{\partial t}= c^2 \frac{\partial^2 v}{\partial x^2}+
	    \frac{2}{B(\alpha,\frac{1}{2})}\int_0^1(1-u^2)^{\alpha-1} F(x,ut) du,\quad \alpha >0\\
	    v(x,0)=0,\\
	    \frac{\partial v}{\partial t}(x,t)\big|_{t=0}=0,
	    \end{cases}
	    \end{equation}
	   by means of the following relationship
	   \begin{equation}\label{for1}
	   v(x,t)=\frac{2}{B(\alpha, \frac{1}{2})}\int_0^1(1-u^2)^{\alpha-1} w(x,ut)du,
	   \end{equation}
	\end{te}
	\begin{proof}
	We recall that the solution of \eqref{for0} is given by
	\begin{equation}
	w(x,t) = \frac{1}{2c}\int_0^t ds \int_{x-c(t-s)}^{x+c(t-s)} F(y,s)dy.
	\end{equation}
	By subsequently deriving $w(x,ut)$ we obtain the useful relations
	\begin{equation}\label{for2}
	\begin{cases}
	\frac{\partial^2 w}{\partial t^2}-c^2u^2\frac{\partial^2 w}{\partial x^2}= u^2 F(x,ut),\\
	\frac{\partial^2 w}{\partial u^2}-c^2t^2 \frac{\partial^2 w}{\partial x^2}= t^2 F(x,ut)\\
	u\frac{\partial w}{\partial u}=t\frac{\partial w}{\partial t}.
	\end{cases}
	\end{equation}
	Then, differentiating \eqref{for1} with respect to time and space, and by using the relationships \eqref{for2} we obtain
	the claimed result.
	\end{proof}
	\begin{os}
    We can give a probabilistic interpretation of the solution \eqref{for1} as
    \begin{equation}\label{bell}
    v(x,t)=\frac{1}{2c}\mathbb{E}\bigg\{\int_0^{\mathfrak{U}(t)}ds\int_{x-c(\mathfrak{U}(t)-s)}
    ^{x+c(\mathfrak{U}(t)-s)}F(y,s)dy \bigg\},
    \end{equation}
    where
    \begin{equation}\label{ao}
    P\{\mathfrak{U}(t)\in du\}/du=\frac{2}{t\,B(\alpha, \frac{1}{2})}\left(1-\frac{u^2}{t^2}\right)^{\alpha-1},
    \end{equation}
    for $0<u<t$. Result \eqref{bell} shows that the solution of the EPD non-homogeneous 
    equation \eqref{for} is the mean value of the solution of the forced wave equation with respect
    to the probability density \eqref{ao}. The relationship between $\mathcal{T}(t)$ and $\mathfrak{U}(t)$ is
    \begin{equation}
    \mathfrak{U}(t)= \frac{|\mathcal{T}(t)|}{c},
    \end{equation}
    where $|\mathcal{T}(t)|$ represents the distance from the starting point 
    of the particle performing the telegraph process.
    \end{os}
	
    \begin{te}
    The solution to the non-homogeneous EPD equation
    \begin{equation}\label{foro}
		    \begin{cases}
		    \frac{\partial^2 v}{\partial t^2}+\frac{2\alpha}{t}\frac{\partial v}{\partial t}= c^2 \frac{\partial^2 v}{\partial x^2}+
		    \frac{2g(x)}{t\,B(\alpha,\frac{1}{2})},\\
		    v(x,0)=0,\\
		    \frac{\partial v}{\partial t}(x,t)\big|_{t=0}=\frac{\Gamma(\alpha+\frac{1}{2})}{\sqrt{\pi}\Gamma(\alpha+1)}g(x),
		    \end{cases}
		    \end{equation}
    has the form 
    \begin{equation}\label{arid}
		   v(x,t)=\frac{2}{B(\alpha, \frac{1}{2})}\int_0^1(1-u^2)^{\alpha-1} w(x,ut)du,
		   \end{equation}
    where $w(x,t)$ is the solution to the Cauchy problem
	\begin{equation}\label{ca0}
			    \begin{cases}
			    \frac{\partial^2 w}{\partial t^2}= c^2 \frac{\partial^2 w}{\partial x^2},\\
			    w(x,0)= 0,\\
			    \frac{\partial w}{\partial t}(x,t)\big|_{t=0}=g(x).
			    \end{cases}
		    \end{equation}
	\end{te}	
    \begin{os}
     The solution \eqref{arid} can be represented as
	 \begin{equation}
			   v(x,t)=\frac{1}{2c}\mathbb{E}\left[\int_{x-c\,\mathfrak{U}(t)}^{x+c\,\mathfrak{U}(t)}g(y)dy\right],
			   \end{equation}
    where $\mathfrak{U}(t)$ is the r.v. with distribution \eqref{ao}.\\
    An alternative representation of \eqref{arid} can be obtained as follows
    \begin{align}
    \nonumber v(x,t)&= \frac{2}{B(\alpha, \frac{1}{2})}\int_0^1(1-u^2)^{\alpha-1} du \frac{1}{2c}\int_{x-ut}^{x+ut}g(y)dy\\
    \nonumber &= \frac{1}{c}\left[\int_{x-ct}^x g(y)\left(1-F\left(\frac{x-y}{ct}\right)\right)dy+\int_{x}^{x+ct} g(y)\left(1-F\left(\frac{y-x}{ct}\right)\right)dy
    \right]\\
    \nonumber &= \frac{1}{c} \left[\int_{x-ct}^x g(y)F\left(\frac{y-x}{ct}\right)dy +
    \int_x^{x+ct}g(y) F\left(\frac{x-y}{ct}\right)dy\right],
    \end{align}
    where 
    \begin{equation*}
    F(z)= \frac{1}{B(\alpha, \frac{1}{2})}\int_{-1}^z(1-u^2)^{\alpha-1}du,
    \end{equation*}
    is the cumulative distribution of the symmetric random variable
    with density
    \begin{equation} 
	f(u)= \frac{1}{B(\alpha, \frac{1}{2})}(1-u^2)^{\alpha-1}, \quad |u|<1.
    \end{equation}
    Consider in the above calculations that for the symmetric r.v. with distribution $F(z)$ and density $f(z)$, we have that
    \begin{equation}\nonumber
    1-F\left(\frac{y-x}{ct}\right) =F\left(\frac{x-y}{ct}\right) 
    \end{equation}
    \end{os}
	
	\section{Other forms of telegraph equations with time-varying coefficients}
	Telegraph processes where the change of directions is governed by non-homogeneous Poisson processes
	lead to a family of telegraph equations with time-varying coefficients of which the EPD equation is a particular case.  
	If the inhomogeneous Poisson process has rate $\lambda(t)$, the distribution $p(x,t)$ of the telegraph process 
	satisfies the following telegraph-type equation (see \cite{kaplan})
	\begin{equation}\label{sec3}
	\begin{cases}
	\frac{\partial^2 p}{\partial t^2}+2\lambda(t)\frac{\partial p}{\partial t}= c^2 \frac{\partial^2 p}{\partial x^2},\\
	p(x,0)=\delta(x),\\
	\frac{\partial p}{\partial t}(x,t)\big|_{t=0}=0.
	\end{cases}
	\end{equation}
	In this section we consider some special forms of $\lambda(t)$ for which the Cauchy problem \eqref{sec3}
	has an explicit solution and then find the law $p(x,t)$ of the corresponding inhomogeneous telegraph process $\mathcal{T}(t)$.
	By means of the exponential transformation
	\begin{equation}\label{exp}
	p(x,t)= e^{-\int_0^t\lambda(s)ds} v(x,t),
	\end{equation}
	we convert the equation in \eqref{sec3} into
	\begin{equation}\label{KG}
	\frac{\partial^2 v}{\partial t^2}-[\lambda'(t)+\lambda^2(t)]v= c^2\frac{\partial^2 v}{\partial x^2}.
	\end{equation}
	In order to find the explicit probability law related to \eqref{sec3}, we assume that
	\begin{equation}\label{Riccati}
	\lambda'(t)+\lambda^2(t)= const. 
	\end{equation}
	Among the solutions of the Riccati equation \eqref{Riccati} we choose the following ones
	\begin{equation*}
	\begin{cases}
	\lambda(t)= \lambda \coth \lambda t,\\
	\lambda(t)= \lambda\tanh \lambda t.
	\end{cases}
	\end{equation*} 
	
	The case where $\lambda(t) = \lambda\tanh(\lambda t)$ was considered by Iacus in \cite{Iacus}.
	In this case the explicit probability law of the related process, namely $Y(t)$, has an absolutely continuous component
	equal to 
	\begin{equation}
	P\bigg\{Y(t) \in dx\bigg\}/dx=\frac{1}{2c \, \cosh \lambda t}\frac{\partial}{\partial t}I_0\left(\frac{\lambda}{c}\sqrt{c^2t^2-x^2}\right)
	    \label{ras}, \quad |x|<ct,
	\end{equation}
	while at $x= \pm ct$ a mass of probability equal to $\frac{1}{2\cosh \lambda t}$ is located.
	We note that the distribution \eqref{ras} can also be derived as follows	
	
	  \begin{align}
	    \nonumber P\{Y(t)\in dx\}&=\frac{dx}{2c \, \cosh \lambda t}\frac{\partial}{\partial t}I_0\left(\frac{\lambda}{c}\sqrt{c^2t^2-x^2}\right)
	    = dx \frac{\lambda t}{2 \, \cosh \lambda t}\frac{I_1\left(\frac{\lambda}{c}\sqrt{c^2t^2-x^2}\right)}{\sqrt{c^2t^2-x^2}}\\
	    \nonumber & =dx\sum_{n=1}^{\infty}\frac{ct(2n)!}{n!(n-1)!}\frac{(c^2t^2-x^2)^{n-1}}{(2ct)^{2n}}
	    \frac{2(\lambda t)^{2n}e^{-\lambda t}}{(1+e^{-2\lambda
	    t})(2n)!}\\
	    \nonumber &= \sum_{n=1}^{\infty}P\{T(t)\in dx|N(t) = 2n\}
	    	    P\{N(t)= 2n|\bigcup_{k=0}^{\infty}\{N(t)= 2k\}\}\\
	    &=P\{T(t)\in dx|\bigcup_{k=0}^{\infty}\{N(t)= 2k\}\} \label{rasal}.
	    \end{align}
	  In \eqref{rasal} we used the conditional distribution of the symmetric telegraph process $T(t)$, $t>0$, that is (see formula (2.18) in 
	  \cite{ales})
	  \begin{equation}
	  P\{T(t)\in dx|N(t)= 2k\}= dx\;\frac{(2k)!}{k!(k-1)!}\frac{ct(c^2t^2-x^2)^{k-1}}{(2ct)^{2k}}, \quad |x|<ct,
	  \end{equation}
	  where $N(t)$, $t>0$ is an homogeneous Poisson process.
	  Note that 
	 \begin{equation}
	     \int_{-ct}^{+ct}P\{Y(t)\in dx \}= 1-\frac{1}{\cosh \lambda t} = 1-P\{N(t)= 0|\bigcup_{k=0}^{\infty}\{N(t)= 2k\}\}.
	 \end{equation}   
	 For the telegraph process with rate $\lambda(t)= \lambda \coth(\lambda t)$ the distribution of the corresponding
	 random motion $X(t)$ with constant velocity $c$ is 
	  \begin{equation}\label{pro1}
	   P\bigg\{X(t)\in dx\bigg\}=\frac{\lambda I_0\left(\frac{\lambda}{c}\sqrt{c^2t^2-x^2}\right)}{2c\sinh \lambda
	     t }dx, \quad |x|<ct.
	   \end{equation}
	   This can be shown by considering that equation \eqref{KG} admits the solution
	   \begin{equation}
	   v(x,t)=I_0\left(\frac{\lambda}{c}\sqrt{c^2t^2-x^2}\right),
	   \end{equation}
	   and thus, by \eqref{exp}, the distribution of $X(t)$ satisfies equation
	   \begin{equation}
	   \frac{\partial^2 p}{\partial t^2}+2\lambda\coth \lambda t\frac{\partial p}{\partial t}= c^2 \frac{\partial^2 p}{\partial x^2},
	   \end{equation}
	   as also a direct check shows.\\
	   The distribution \eqref{pro1} has no discrete component at $x = \pm ct$ as in the EPD case, because the rate
	   is infinitely large for $t\rightarrow 0^+$. 
	   		In Figure 1 we compare the behaviour of the distributions \eqref{ras} and \eqref{pro1}. They are similar but the density of $X(t)$
	   		has a maximum exceeding that of $Y(t)$ because the last one has part of the distribution concentrated on the endpoints $x=\pm ct$.
	   		\begin{figure}
	   		            \centering
	   		            \includegraphics[scale=.31]{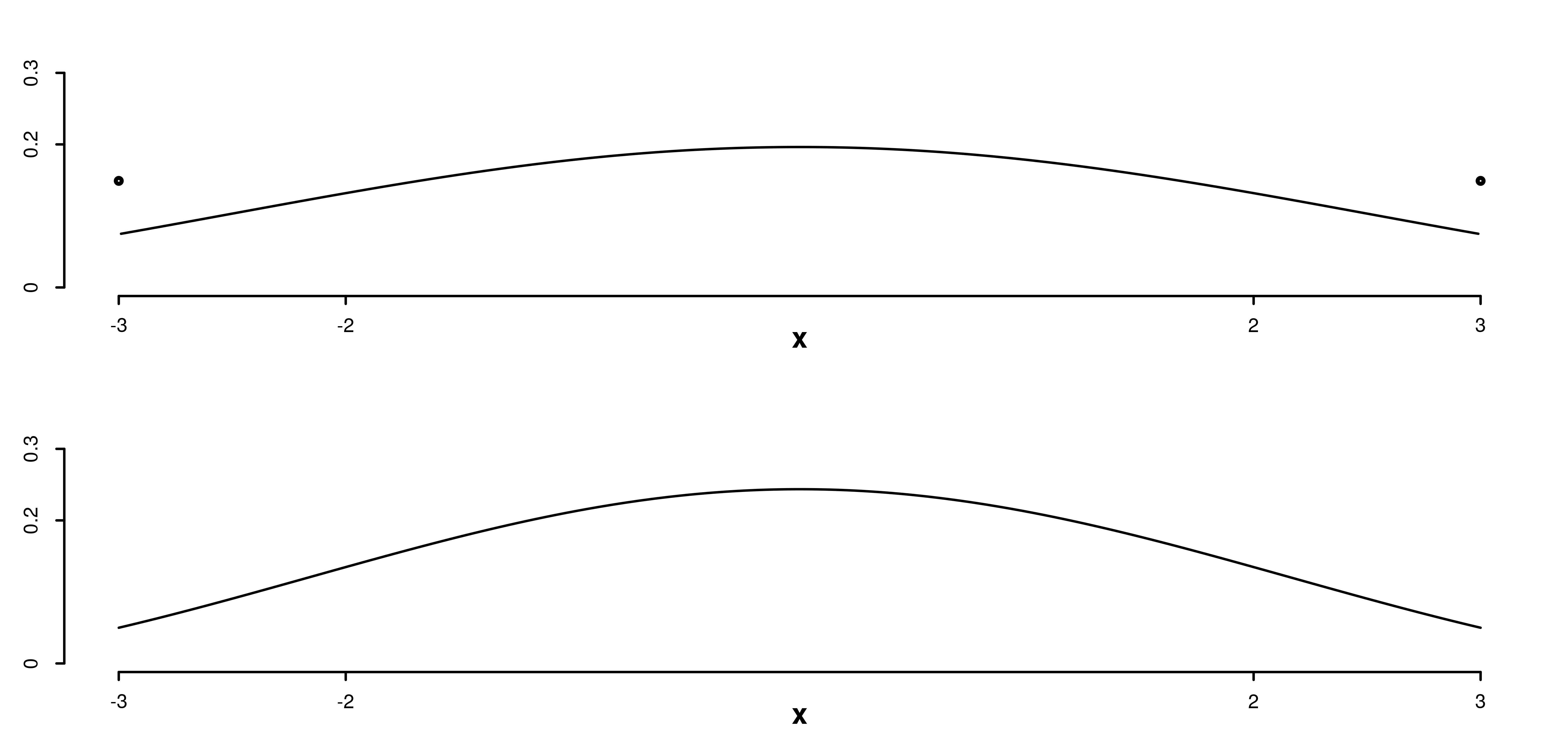}
	   		            \caption{Plots of the distributions of the processes $Y(t)$ (figure above) and $X(t)$ (figure below)
	   		            for $t= 3$, $c= \lambda = 1$.}
	   		    \end{figure}

	   We give two further derivations of \eqref{pro1}
	   in the next two remarks
	   \begin{os}
	   In \cite{SPA} it was shown that a random flight in $\mathbb{R}^3$ with intermediate Dirichlet
	   distributed displacements, the position of
	   the moving particle $\mathbf{X}_{3}(t)$ has distribution
	     \begin{align}\label{volo}
	    	         \frac{P\{\mathbf{X}_{3}(t)\in d\mathbf{x}_3\}}{\prod_{j=1}^{3} dx_j}=
	    	         \left(\frac{\lambda}{2c}\right)^2\frac{1}{\pi \sinh \lambda t}
	    	        \frac{I_1\left(\frac{\lambda}{c}\sqrt{c^2t^2-\|\mathbf{x}_3\|^2}\right)}{\sqrt{c^2t^2-\|\mathbf{x}_3\|^2}},
	    	        \, \|\mathbf{x}_3\|^2< c^2t^2.
	    	        \end{align}
	   The distribution \eqref{volo} has a singular component on the sphere 
	   \begin{equation*}
	   S_{ct}= \bigg\{\mathbf{x}\in \mathbb{R}^3: \|\mathbf{x}_3\|^2= c^2t^2 \bigg\}
	   \end{equation*}
	   	         equal to
	   	        \begin{equation}\label{sph}
	   	        \int_{S^3_{ct}} P\{\mathbf{X}_{3}(t)\in d\mathbf{x}_3\}=1-\frac{\lambda t}{\sinh \lambda t}=
	   	        1-P\{N(t)= 0|\bigcup_{k=0}^{\infty}\{N(t)= 2k+1\}\},
	   	        \end{equation}
	   	   We observe that the projection of $\mathbf{X}_3(t)$ on the line, say $X_1(t)$, has distribution
	   	   coinciding with \eqref{pro1} (see \cite{SPA}).
	        
	   \end{os}
	   
	   \begin{os}
	   Another derivation of the distribution \eqref{pro1}, similar to \eqref{rasal} reads
	    \begin{align}
	       \nonumber P\{X(t)\in dx\}&= \sum_{n=0}^{\infty}P\{X(t)\in dx|N(t) = 2n+1\}
	       P\{N(t)= 2n+1|\bigcup_{k=0}^{\infty}\{N(t)= 2k+1\}\}\\
	       \nonumber & =\sum_{n=0}^{\infty}\frac{(2n+1)!}{(n!)^2}\frac{(c^2t^2-x^2)^n}{(2ct)^{2n+1}}
	       \frac{2(\lambda t)^{2n+1}e^{-\lambda t}}{(1-e^{-2\lambda
	       t})(2n+1)!}\\
	       &= \frac{\lambda}{2c}\frac{I_0\left(\frac{\lambda}{c}\sqrt{c^2t^2-x^2}\right)}{\sinh \lambda
	       t}\label{riga},
	       \end{align}
		\end{os}
		\begin{os}
		The distributions \eqref{ras} and \eqref{pro1} are related to that of the classical telegraph 
		process $T(t)$ by the relationship
		\begin{align}
		\nonumber P\bigg\{T(t)\in dx\bigg\}&= P\bigg\{T(t)\in dx|N(t) =\bigcup_{k=0}^{\infty}\{N(t)= 2n\}\bigg\}
			       P\bigg\{\bigcup_{k=0}^{\infty}\{N(t)= 2n\}\bigg\}\\
		\nonumber &+P\bigg\{T(t)\in dx|N(t) =\bigcup_{k=0}^{\infty}\{N(t)= 2n+1\}\bigg\}
					       P\bigg\{\bigcup_{k=0}^{\infty}\{N(t)= 2n+1\}\bigg\}\\
		\nonumber & =P\bigg\{Y(t)\in dx\bigg\}P\bigg\{\bigcup_{k=0}^{\infty}\{N(t)= 2n\}\bigg\}\\
				 \nonumber & +P\bigg\{X(t)\in dx\bigg\}P\bigg\{\bigcup_{k=0}^{\infty}\{N(t)= 2n+1\}\bigg\}   \\
		\nonumber & =   dx\frac{e^{-\lambda t}}{2c}\left[\lambda I_0\left(\frac{\lambda}{c}
				    \sqrt{c^2t^2-x^2}\right)+\frac{\partial}{\partial t}I_0\left(\frac{\lambda}{c}
				    \sqrt{c^2t^2-x^2}\right)
				    \right], \qquad  |x|<ct. 
		\end{align}
		\end{os}

    \section{Random flights in $\mathbb{R}^d$ governed by EPD equations}

    On the basis of the analysis performed in the previous section,
    here we consider random flights in $\mathbb{R}^d$,
    $d\geq 2$, whose probability laws are governed by
    $d$-dimensional EPD equations.
    By $\mathbf{X}_d(t)$, $t>0$ we denote the position of the particle performing the random motion considered here.
     We start with the following
    \begin{te}\label{fera}
    The joint probability law 
    \begin{equation}\label{epo}
        p(\mathbf{x}_d,t) =
        \frac{\Gamma(\gamma+\frac{d}{2})}{\pi^{d/2}\Gamma(\gamma)}\frac{1}{(ct)^{d-2+2\gamma}}\left(c^2t^2-\|\mathbf{x}_d\|^2\right)^{\gamma-1},
        \quad \|\mathbf{x}\|^2< c^2t^2,\; \gamma>0,
        \end{equation}
        where $\mathbf{x}\equiv (x_1, x_2, \dots, x_d)$, $d\in \mathbb{N}$, of the random vector $\mathbf{X}_d(t)$, $t>0$,    
    solves the Euler-Poisson-Darboux equation
    \begin{equation}\label{epos}
    \left(\frac{\partial^2}{\partial t^2}+\frac{d+2\gamma-1}{t}\frac{\partial}{\partial
    t}\right)p(\mathbf{x}_d,t)= c^2 \displaystyle \sum_{j=1}^{d}\frac{\partial^2}{\partial x_j^2} p(\mathbf{x}_d,t).
    \end{equation}
    \end{te}
    \begin{proof}
    By direct calculations one can ascertain that \eqref{epo} satisfies \eqref{epos}.
    \end{proof}
    In order to prove that \eqref{epo} is a probability law, we should perform an integration over
    the hyper-sphere $S^d_{ct}$ of radius $ct$ and again by direct calculation we obtain
    the claimed result.
 
    \begin{os}
    We observe that, for $\gamma =\frac{n}{2}(d-1)$, \eqref{epo} formally coincides with
    the conditional distributions of random flights with Dirichlet displacements
    studied by De Gregorio and Orsingher in \cite{Ale}, for any $d\geq 2$. For $d=1$ it coincides with
    the result of Foong and van Kolck \cite{Foong} (see the previous section).\\
    Moreover if $d=2$ and $\gamma = \frac{n\alpha}{2}$, \eqref{epo} coincides with the conditional distribution of
    the fractional generalization of the finite velocity planar random motions with an infinite number
    of directions recently studied in \cite{noi}.
    \end{os}
  
    In the general case it is simple to compute the marginals of the distribution $p(\mathbf{x}_d,t)$, as shown in the following Theorem
    \begin{te}
    The projection of the process $\mathbf{X}_d(t)$, $t>0$, onto a lower space of dimension $m$, leads to the following
    distribution
    \begin{equation}\label{marg}
    f_{\mathbf{X}_m}^d(\mathbf{x}_m,t)= \frac{\Gamma(\gamma+\frac{d}{2})}{\Gamma(\gamma+\frac{d}{2}-\frac{m}{2})}
    \frac{(c^2t^2-\|\mathbf{x}_m\|^2)^{\gamma-1+\frac{d-m}{2}}}{\pi^{m/2}(ct)^{2\gamma+d-2}},
    \end{equation} 
    with $\|\mathbf{x}_m\|<ct$.
    \end{te}
    \begin{proof}
    The projection of the random process $\mathbf{X}_d(t)$ onto the space $\mathbb{R}^m$, with $m<d$, represents
    a random flight with $m$ components having probability density given by
    \begin{align}
        f_{\mathbf{X}_m}^d(\mathbf{x}_m,t)= \int_{-\sqrt{c^2t^2-\|\mathbf{x}_m\|^2}}^{\sqrt{c^2t^2-\|\mathbf{x}_m\|^2}}dx_{m+1}\dots
        \int_{-\sqrt{c^2t^2-\|\mathbf{x}_{d-1}\|^2}}^{\sqrt{c^2t^2-\|\mathbf{x}_{d-1}\|^2}}p_{\mathbf{X}_d}(\mathbf{x}_d,t)dx_d.
    \end{align}
    The result \eqref{marg} emerges by successively integrating the density \eqref{epo} and by suitably 
    manipulating the Beta integrals emerging in the calculations.
    \end{proof}
    The law of the marginal \eqref{marg} is a solution to the EPD equation 
    \begin{equation}\label{dome}
    \left(\frac{\partial^2}{\partial t^2}+\frac{d+2\gamma-1}{t}\frac{\partial}{\partial t}\right)p(\mathbf{x}_m,t)=
    c^2\sum_{j=1}^m \frac{\partial^2}{\partial x_j^2}p(\mathbf{x}_m,t),
    \end{equation}
    as a direct check proves. The same conclusion can be reached by observing that \eqref{marg}
    has the same form as \eqref{epo} with $\gamma$ replaced throughout by $\gamma+\frac{d-m}{2}$.
    The same check can be done on equation \eqref{dome} which can be derived from \eqref{epos} in the same way.
   
    On the basis of the formal analogy between the conditional distributions of random flights with Dirichlet displacements studied in \cite{Ale} and the probability distribution \eqref{epo} we have the following
    \begin{te}
    For the vector process $\mathbf{X}_{d}(t)= (X_1(t), \dots, X_d(t))$, $t>0$, with distribution
    \eqref{epo}, the characteristic function is
    \begin{equation}\label{ftra}
    \mathbb{E}\{e^{i(\boldsymbol{\alpha}\cdot \mathbf{X}_d(t))}\}=\frac{2^{\gamma+\frac{d}{2}-1}
    \Gamma(\gamma+\frac{d}{2})}{(ct\|\boldsymbol{\alpha}\|)^{\gamma+\frac{d}{2}-1}}J_{\gamma+\frac{d}{2}-1}
    \left(ct\|\boldsymbol{\alpha}\|\right),
    \end{equation}
    where $\boldsymbol{\alpha}= (\alpha_1, \dots, \alpha_d)\in \mathbb{R}^d$ and
    \begin{equation*}
    J_{\nu}(x)= \sum_{k=0}^{\infty}(-1)^k\left(\frac{x}{2}\right)^{2k+\nu}\frac{1}{k!\Gamma(k+\nu+1)},
    \end{equation*}
    is the Bessel function of order $\nu\in \mathbb{R}$.
    \end{te}
    This is a simple consequence of Theorem 1, pag.683, in \cite{Ale}.
    As a corollary, we have that the Fourier transform of the fundamental solution of
    the $d$-dimensional EPD equation \eqref{epos} is given by \eqref{ftra}.
    \begin{os}
   	We observe that the fractional version of the EPD equation \eqref{epos}
   	\begin{equation}
   	\left(\frac{1}{t^{2\lambda}}\frac{\partial}{\partial t}t^{2\lambda}\frac{\partial}{\partial t}\right)^{\beta}p(\mathbf{x}_d,t)=
   	c^2 \displaystyle \sum_{j=1}^{d}\frac{\partial^2}{\partial x_j^2} p(\mathbf{x}_d,t),
   	\quad \beta \in (0,1)
   	\end{equation}
	can be treated by means of the McBride theory of fractional powers of hyper-Bessel operators, see \cite{noi, mc}.	
    \end{os}
    
     \subsection{Random flights with random velocities}
    
        In this section we consider planar random motions with random
        velocities. Our construction is based on the marginal
        distributions of the projection of random flights with Dirichlet displacements in
        $\mathbb{R}^d$ onto $\mathbb{R}^m$ considered in \cite{Ale}. We will consider 
        for semplicity the case in which the projection is onto the plane, i.e. $m=2$.
        In this case, the authors have shown that the marginal distributions of the projections of the
        processes $\mathbf{X}_{d}(t)$ and $\mathbf{Y}_{d}(t)$, $t>0$, onto
        $\mathbb{R}^2$ are given by (see Theorem 4 pag.695)
        \begin{align}
        \label{ciccio} &f^d_{\mathbf{X}_2}(\mathbf{x}_2,t;n)= \frac{\Gamma\left(\frac{n+1}{2}(d-1)+
        \frac{1}{2}\right)}{\Gamma\left(\frac{(n+1)}{2}(d-1)-\frac{1}{2}\right)}
        \frac{(c^2t^2-\|\mathbf{x}_2\|^2)^{\frac{n+1}{2}(d-1)-\frac{3}{2}}}{\pi(ct)^{(n+1)(d-1)-1}},
        \quad &d\geq 2,\\
        \label{ci} &f^d_{\mathbf{Y}_2}(\mathbf{y}_2,t;n)= \frac{\Gamma\left((n+1)(\frac{d}{2}-1)
        +1\right)}{\Gamma\left((n+1)(\frac{d}{2}-1)\right)}
        \frac{(c^2t^2-\|\mathbf{y}_2\|^2)^{(n+1)(\frac{d}{2}-1)-1}}{\pi(ct)^{2(n+1)(\frac{d}{2}-1)}},
        \quad &d\geq 3
        \end{align}
        with $\|\mathbf{x}_2\|<ct$ and $\|\mathbf{y}_2\|<ct$.
        We note that all distributions \eqref{ciccio} and \eqref{ci} are of the form \eqref{epo}
        and for special values of $\gamma$ provide us a probabilistic interpretation of the distributions obtained
        in Theorem \ref{fera} as random walks in the space $\mathbb{R}^2$ with Dirichlet distributed 
        displacements. For $m<d$ these random flights can be regarded as random motions with random velocities.
        We examine here the special case $m= 2$ to have a flash of insight of what is going on.
       According to the distribution \eqref{ciccio} (and in an analogous way for
        \eqref{ci}) the moving particle in $\mathbb{R}^2$ has coordinates 
        \begin{equation}\label{fera1}
        \begin{cases}
        X(t)= \displaystyle\sum_{j=1}^{n+1}\tau_j v(\theta_{1,j}, \dots, \theta_{d-2,j})\cos \phi_j\\
        \\
        Y(t)= \displaystyle\sum_{j=1}^{n+1} \tau_j v(\theta_{1,j}, \dots, \theta_{d-2,j})\sin
        \phi_j,
        \end{cases}
        \end{equation}
        where $0<\theta_k< \pi$, $0\leq \phi_k \leq 2\pi$ for $k=1, 2, \dots,
        d-2$ and
        \begin{equation}
        v(\theta_{1,j}, \dots, \theta_{d-2,j})= c \sin\theta_{1,j}
        \sin\theta_{2,j}\dots \sin\theta_{d-2,j},
        \end{equation}
        is the random
        velocity of the $j$-th displacement.  The couple \eqref{fera1}
        represents the position of the shadow of the moving particle onto $\mathbb{R}^2$
        after $n+1$ displacements.
        The vector $(\tau_1, \dots, \tau_{n+1})$ represents
        the time length between successive changes of direction and
        $\tau_j v(\theta_{1,j}, \dots, \theta_{d-2,j})$ is the length of
        the $j$-th displacement. The distribution of $\tau_1,
        \dots,\tau_{n+1}$ is Dirichlet while the probability density of the r.v. $(\theta_1,\dots,
        \theta_{d-2})$, $0\leq \theta_j\leq \pi$, $1\leq j\leq d-2$ is given by
        \begin{equation}
        p(\theta_1, \theta_2, \dots,
        \theta_{d-2})=\frac{\Gamma\left(\frac{d}{2}\right)}{(2\pi)^{\frac{d-2}{2}}}
        \sin^{d-2}\theta_1\sin^{d-3}\theta_2\dots \sin
        \theta_{d-2}, \quad d\geq 2.
        \end{equation}
        The r.v. $\phi$ is uniform in $[0,2\pi]$ and independent from
        $\theta_1, \dots, \theta_{d-2}$ and $\tau_1, \dots,
        \tau_{n+1}$.\\
        If $\tau_j = s_j-s_{j-1}, \; j=1, \dots, n+1$ and $s_1, \dots,
        s_n$ are uniformly distributed in the simplex, that is
        \begin{equation}
        f(s_1, \dots, s_{n+1})= \frac{n!}{t^n}ds_1\dots ds_n, \quad
        0<s_1<\dots<s_n<t,
        \end{equation}
        and $v(\theta_{1,j}, \dots, \theta_{d-2,j})= c$, then we obtain the
        particular model considered in \cite{kol} and \cite{sta}.
        Then we can calculate, for example, the mean velocity
        \begin{align}
        \nonumber &\mathbb{E}v(\theta_1, \dots, \theta_{d-2})\\
        \nonumber &=
        \frac{c\Gamma\left(\frac{d}{2}\right)}{\pi^{\frac{d-2}{2}}}
        \int_0^\pi d\theta_1\int_0^\pi d\theta_2\dots \int_0^\pi
        d\theta_{d-2}\left(\sin\theta_1
        \sin\theta_2\dots \sin\theta_{d-2}\right)\left(\sin^{d-2}\theta_1\sin^{d-3}\theta_2\dots \sin
        \theta_{d-2}\right)\\
        \nonumber &=\frac{2^{d-2} c\Gamma\left(\frac{d}{2}\right)}{\pi^{\frac{d-2}{2}}}
        \prod_{j=1}^{d-2}\int_0^{\pi/2}d\theta_j\sin^{d-j} \theta_j\\
        \nonumber &=\frac{c\Gamma\left(\frac{d}{2}\right)\Gamma\left(\frac{3}{2}\right)}{\pi^{\frac{d-2}{2}}}
        \frac{\pi^{\frac{d-2}{2}}}{\Gamma\left(\frac{d+1}{2}\right)}=
        \frac{c\Gamma\left(\frac{d}{2}\right)\Gamma\left(\frac{3}{2}\right)}{\Gamma\left(\frac{d+1}{2}\right)}.
        \end{align}
        
        We observe that this random motion slightly differs from the one
        studied by Kolesnik and Orsingher in \cite{kol}. On the other
        hand we observe that, for $d= 2$ the conditional distribution
        \eqref{ciccio} coincides with \eqref{coko} below and the mean velocity
        in this case exactly coincides with the constant velocity $c$ of
        the random motion. In the general case the mean velocity is a fraction of the velocity $c$, depending
        by the dimension $d$ of the space. This is due to the fact that the random velocity reaches its maximum for
        $\theta = \pi/2$ where it coincides with $c$, but for all the other values of $\theta$ is clearly less than
        $c$.\\
        \begin{os}
        We can consider more general planar random motions with random
        velocities, simply by assuming a different distribution $p(\theta_1, \theta_2, \dots,
        \theta_{d-2})$ of the r.v. $(\theta_1,\dots,
        \theta_{d-2})$. In this case we will obtain 
        conditional distributions different from \eqref{ciccio} and
        \eqref{ci}. We also remark that with a similar approach it is
        possible to study random flights with random velocities
        in $\mathbb{R}^3$, which is of the most interest for physical applications.
        \end{os}
        
        We now consider inhomogeneous planar motions with conditional distributions \eqref{ciccio} and
        \eqref{ci}, whose changes of directions are driven by an
        inhomogeneous Poisson processes with rate function $\lambda (t)$. In the general case, we find
        the following unconditional distributions:
        \begin{align}
        \label{cir}&p_{d,2}(\mathbf{x}_2,t)=e^{-\Lambda(t)}\sum_{n=0}^{\infty}\frac{\Gamma\left(\frac{n+1}{2}(d-1)+
        \frac{1}{2}\right)}{\Gamma\left(\frac{(n+1)}{2}(d-1)-\frac{1}{2}\right)}
        \frac{(c^2t^2-\|\mathbf{x}_2\|^2)^{\frac{n+1}{2}(d-1)-\frac{3}{2}}}{\pi(ct)^{(n+1)(d-1)-1}}
        \frac{\left(\Lambda(t)\right)^n}{n!}\\
        \label{cir1}&p_{d,2}(\mathbf{y}_2,t)=e^{-\Lambda(t)}\sum_{n=0}^{\infty}\frac{\Gamma\left((n+1)(\frac{d}{2}-1)
        +1\right)}{\Gamma\left((n+1)(\frac{d}{2}-1)\right)}
        \frac{(c^2t^2-\|\mathbf{y}_2\|^2)^{(n+1)(\frac{d}{2}-1)-1}}{\pi(ct)^{2(n+1)(\frac{d}{2}-1)}}
        \frac{\left(\Lambda(t)\right)^n}{n!}.
        \end{align}
        
        We observe that in this case the density laws of the planar random processes are absolutely
        continuous because they are obtained by projection of random flights
        with Dirichlet displacements onto $\mathbb{R}^2$. We are going to
        consider some special cases of $\lambda(t)$, when the densities
        can be written down as a sum of exponential functions.
        In particular, let us consider the interesting case of the
        projection of the random flight from $\mathbb{R}^3$ onto
        $\mathbb{R}^2$. Moreover, we choose a particular inhomogeneous
        Poisson process with $\Lambda(t)= (\lambda t)^2$ in order to
        find  an exponential form of the probability
        distribution. Thus, from \eqref{cir} under these assumptions we have that
        \begin{align}
        \nonumber p_{3,2}(\mathbf{x}_2,t)=\left(\frac{1}{2}+\frac{\lambda^2}{c^2}(c^2t^2-\|\mathbf{x}_2\|^2)\right)
        \frac{\exp\left(-\frac{\lambda^2}{c^2}\|\mathbf{x}_2\|^2\right)}{\pi ct
        \sqrt{c^2t^2-\|\mathbf{x}_2\|^2}},
        \end{align}
        and
        \\
        \begin{align}\nonumber
        p_{3,2}(\mathbf{y}_2,t)=\left(1+\frac{\lambda^2 t}{c}(c^2t^2-\|\mathbf{y}_2\|^2)\right)
                \frac{\exp\left(-\lambda^2t^2+\frac{\lambda^2 t}{c}(\sqrt{c^2t^2-\|\mathbf{y}_2\|^2})\right)}{2\pi ct
                \sqrt{c^2t^2-\|\mathbf{y}_2\|^2}}.\label{stan}
        \end{align}

	\section{Planar random motions directed by inhomogeneous Poisson processes}

    \subsection{Planar random motions with four orthogonal directions}

    In this section 
    we show that for a non-homogeneous Poisson process with an arbitrary rate function $\lambda(t)$, the probability law of a finite-velocity
    planar motion with four orthogonal directions can be obtained by solving a system of four 
    differential equations with time-varying coefficients.  
    We here adopt the approach and notation of Orsingher and Kolesnik \cite{koll} which we briefly describe.
    The possible directions of motion are represented by the vectors 
    \begin{equation}\nonumber
    v_k = \left(\cos\frac{\pi k}{4}, \sin \frac{\pi k}{4}\right), \quad k = 1,3,5,7,
    \end{equation}
 	which for the sake of simplicity are denoted by means of the couples $(++)$, $(-+)$, $(+-)$ and $(--)$, called also polarities of motion.
	We assume that the velocity has horizontal and vertical components equal to $c/2$.
	The non-homogeneous Poisson process governs the changes of polarities with the following rule. At the occurrence of
	each event, the particle starts moving in the orthogonal direction and chooses each orientation 
	with equal probability. For example, if the current polarity is $(++)$ then, after the Poisson event, the moving particle
	either chooses $(-+)$ or $(+-)$ with probability $1/2$.  

     Let $(\mathcal{X}(t),\mathcal{Y}(t))$ be the random vector representing the particle position on the plane and
	$D(t)$ denote the polarity of its motions at time $t$.
	The functions representing the joint distributions of the particle position $(\mathcal{X}(t),\mathcal{Y}(t))$ and of the polarity 
	$D(t)$
	\begin{equation}
	f_{a\;b}(x,y,t)dxdy = P\{\mathcal{X}(t)\in dx, \mathcal{Y}(t)\in dy, D(t)= (a\;b)\},
	\end{equation}   
	where $(a \; b)$ represent the four couples of possible directions,
	satisfy the following system of partial differential equations
    \begin{align}
      \label{17} &\frac{\partial f_{++}}{\partial t}= -\frac{c}{2}\frac{\partial f_{++}}{\partial x}-\frac{c}{2}\frac{\partial f_{++}}{\partial y}+
       \frac{\lambda(t)}{2}(f_{-+}+f_{+-}-2f_{++})\\
      \nonumber &\frac{\partial f_{+-}}{\partial t}= -\frac{c}{2}\frac{\partial f_{+-}}{\partial x}+\frac{c}{2}\frac{\partial f_{+-}}{\partial y}+
                  \frac{\lambda(t)}{2}(f_{--}+f_{++}-2f_{+-})\\
      \nonumber &\frac{\partial f_{-+}}{\partial t}= \frac{c}{2}\frac{\partial f_{-+}}{\partial x}-\frac{c}{2}\frac{\partial f_{-+}}{\partial y}+
                  \frac{\lambda(t)}{2}(f_{++}+f_{--}-2f_{-+})   \\
     \nonumber    &\frac{\partial f_{--}}{\partial t}= \frac{c}{2}\frac{\partial f_{--}}{\partial x}+\frac{c}{2}\frac{\partial f_{--}}{\partial y}+
                         \frac{\lambda(t)}{2}(f_{+-}+f_{-+}-2f_{--}).   
       \end{align}
      
       Then, by introducing the auxiliary functions
       \begin{align}
   		& p = f_{++}+f_{--}+f_{-+}+f_{+-}\quad w = f_{++}+f_{+-}-f_{-+}-f_{--}\\
   		\nonumber & z = f_{++}-f_{+-}+f_{-+}-f_{--}\quad u = f_{++}-f_{+-}-f_{-+}+f_{--}, 	    
       \end{align}
       we obtain the following system of coupled partial differential equations with time-varying coefficients
       \begin{equation}\label{sis}
       \begin{cases}
       \frac{\partial p}{\partial t}= -\frac{c}{2}\frac{\partial w}{\partial x}-\frac{c}{2}\frac{\partial z}{\partial y}\\
       \frac{\partial z}{\partial t}= -\frac{c}{2}\frac{\partial u}{\partial x}-\frac{c}{2}\frac{\partial p}{\partial y}-\lambda(t)z\\
       \frac{\partial w}{\partial t}= -\frac{c}{2}\frac{\partial p}{\partial x}-\frac{c}{2}\frac{\partial u}{\partial y}-\lambda(t)w\\
       \frac{\partial u}{\partial t}= -\frac{c}{2}\frac{\partial z}{\partial x}-\frac{c}{2}\frac{\partial w}{\partial y}-2\lambda(t)u.
       \end{cases}
       \end{equation}
       Then, in order to find the explicit form the probability law $p(x,y,t)$ of the random motion, one has to solve the system of  partial differential equations with time-varying coefficients
       \eqref{sis} which clearly depends on the particular choice of $\lambda(t)$. 
       Let us denote by $p(x,y,t)$ the absolutely continuous component of of the probability density 
       \begin{equation*}
       F(x,y,t)= P\{\mathcal{X}(t)\leq x, \mathcal{Y}(t)\leq y\}.
       \end{equation*}
       
       We are now able to state the following
       \begin{te}
       The function $p(x,y,t)$ satisfies the hyperbolic partial differential equation with time-varying coefficients
       \begin{align}\label{sis0}
      \frac{\partial^4 p}{\partial t^4}=& -4\lambda\frac{\partial^3 p}{\partial t^3}+ \left(
      \frac{c^2}{2}\Delta-5\lambda^2-4\frac{d \lambda}{dt}
      \right)\frac{\partial^2 p}{\partial t^2}\\
      &\nonumber +\left(\lambda c^2\Delta-5\lambda\frac{d\lambda}{d t}-2\lambda^3 -\frac{d^2 \lambda}{dt^2}\right)\frac{\partial p}{\partial t}\\
      \nonumber &-\frac{c^4}{2^4}\left(\frac{\partial^2}{\partial x^2}-\frac{\partial^2}{\partial y^2}\right)^2 p+\frac{ c^2}{2}\left(\lambda^2+\frac{d\lambda}{dt}\right)\Delta p
       \end{align} 
       \end{te}
 		\begin{proof}
 		We follow the main steps of the proof of Theorem 2 in \cite{koll}. First of all, by differentiating the first equation of 
 		\eqref{sis} with respect to $t$ and then inserting the second equation differentiated with respect to $x$ and
 		the third one differentiated with respect to $y$, we obtain
 		\begin{equation}\label{sis1}
 		\frac{\partial^2 p}{\partial t^2}= \frac{c^2}{4}\Delta p-\lambda\frac{\partial p}{\partial t}+\frac{c^2}{2}\frac{\partial^2 u}{\partial x\partial y}.
 		\end{equation}
 		By differentiating \eqref{sis1} with respect to $t$ and by substituting in it the fourth equation of \eqref{sis} differentiated
 		with respect to $x$ and $y$ and then using \eqref{sis1} we have that
 		\begin{equation}\label{sis2}
 		\frac{\partial^3 p}{\partial t^3}= \frac{c^2}{4}\Delta \frac{\partial p}{\partial t}
 		-3\lambda\frac{\partial^2 p}{\partial t^2}+\frac{\lambda c^2}{2}\Delta p-2\lambda^2 \frac{\partial p}{\partial t}
 		-\frac{d \lambda}{d t}\frac{\partial p}{\partial t}-\frac{c^3}{4}\left(\frac{\partial^3 z}{\partial x^2\partial y}
 		+\frac{\partial^3 w}{\partial y^2 \partial x}\right).
 		\end{equation} 
 		Finally, differentiating \eqref{sis2} with respect to $t$ and using \eqref{sis1} and \eqref{sis2}, we obtain the claimed result.
 		\end{proof}  
 		In the case $\lambda = const.$, equation \eqref{sis0}
 		coincides with (3.3) of Theorem 2 of \cite{koll}, that is
 		\begin{equation}
 		\left(\frac{\partial}{\partial t}+\lambda\right)^2
 		\left(\frac{\partial^2}{\partial t^2}+2\lambda \frac{\partial}{\partial t}-\frac{c^2}{2}\Delta\right)p
 		+\frac{c^4}{2^4}\left(\frac{\partial^2}{\partial x^2}-\frac{\partial^2}{\partial y^2}\right)^2 p=0
 		\end{equation}
 		\begin{os}
 		By means of the rotation $u = y+x$, $v= y-x$, the equation \eqref{sis0} takes the form
 		\begin{align}\label{jappo}
 		&\frac{\partial^{4}p}{\partial t^4}= -4\lambda \frac{\partial^3 p}{\partial t^3}+
 		\left(c^2\bigg\{\frac{\partial^2}{\partial u^2}+\frac{\partial^2}{\partial v^2}\bigg\}-5\lambda^2
 		-4\frac{d\lambda}{dt}\right)\frac{\partial^2 p}{\partial t^2}\\
 		\nonumber &+\left(2\lambda c^2\bigg\{\frac{\partial^2}{\partial u^2}+\frac{\partial^2}{\partial v^2}\bigg\}
 		-5\lambda \frac{d\lambda}{dt}-2\lambda^3-\frac{d^2\lambda}{dt^2}\right)\frac{\partial p}{\partial t}\\
 		\nonumber &-c^4\frac{\partial^4p}{\partial u^2\partial v^2}+c^2\left(\lambda^2+\frac{d\lambda}{dt}\right)
 		\bigg\{\frac{\partial^2}{\partial u^2}+\frac{\partial^2}{\partial v^2}\bigg\}p.
 		\end{align}
 		For $\lambda = const.$, equation \eqref{jappo}
 		coincides with the equation (3.9) of \cite{Ors}.
 		In this particular case, the law of the moving particle
 		$(\mathcal{X}(t), \mathcal{Y}(t)$ coincides with
 		that of the random vector
 		\begin{equation}
 		\begin{cases}
 		\mathcal{X}(t)= \mathcal{U}(t)+\mathcal{V}(t)\\
 		\mathcal{Y}(t)= \mathcal{U}(t)-\mathcal{V}(t),
 		\end{cases}
 		\end{equation} 
 		where $\mathcal{U}$, $\mathcal{V}$, are independent 
 		telegraph processes with parameters $c/2$ and $\lambda/2$.
 		\end{os}

    \subsection{Inhomogeneous planar random motions with infinite directions}

    In this section we consider planar random motions driven by an
    inhomogeneous Poisson process. In \cite{kol}, the authors
    studied the random motion of a particle starting its
    motion from the origin of the plane with initial direction
    $\theta$ uniformly distributed in $[0,2\pi)$. In their model the changes of
    direction of the particle are driven by an homogeneous Poisson process with rate
    $\lambda >0$. This means that at the instants of occurrence of an homogeneous Poisson process, the particle takes
    a new direction with uniform distribution in $[0,2\pi)$,
    independently from the previous direction. The main aim of this
    section is to consider different models of planar random motions
    with finite velocity where changes of direction are driven by
    inhomogeneous Poisson processes.
    We recall the
    following result from \cite{kol} (Theorem 1, pag. 1173), that is the starting point of
    our investigation.
    \begin{te}\label{kole}
    For all $n\geq 1$, $t>0$ we have that the conditional
    distribution is given by
    \begin{equation}\label{coko}
    P\{X(t)\in dx, Y(t)\in dy| N(t)= n\}=
    \frac{n}{2\pi(ct)^n}[c^2t^2-(x^2+y^2)]^{\frac{n}{2}-1}dxdy,
    \end{equation}
    where $(x,y)\in int C_{ct}$, with $C_{ct}= \{(x,y)\in \mathbb{R}^2:x^2+y^2\leq
    c^2t^2\}$.
    \end{te}
    Inspired by the previous section we now consider three
    different cases, corresponding to different choices of the rate of the Poisson process. \\
    We first consider the case where the particle changes direction
    only at odd-order Poisson times. Then we have
    \begin{te}
    The distribution of the planar
    random motion $(X(t),Y(t))$, when the changes of direction are
    taken only at odd-order Poisson times is given by
    \begin{equation}\label{pover}
    P\{X(t)\in dx, Y(t)\in dy\}=\frac{\lambda}{2\pi c}\frac{1}{\sinh(\lambda t)}
    \frac{\cosh\left(\frac{\lambda}{c}\sqrt{c^2t^2-(x^2+y^2)}\right)}{\sqrt{c^2t^2-(x^2+y^2)}}dxdy,
    \end{equation}
    where $(x,y)\in int\, C_{ct}$.
    \end{te}

    \begin{proof}
    In view of Theorem \ref{kole} we have
    \begin{align}
    \nonumber &P\{X(t)\in dx, Y(t) \in dy\} \\
    \nonumber &=\sum_{n=0}^{\infty}P\{X(t)\in dx, Y(t)\in dy|N(t) = 2n+1\}
    P\{N(t)= 2n+1|\bigcup_{k=0}^{\infty}\{N(t)= 2k+1\}\}\\
    \nonumber &= \sum_{n=0}^{\infty}\frac{(2n+1)}{2\pi
    (ct)^{2n+1}}\left[\sqrt{c^2t^2-(x^2+y^2)}\right]^{2n-1}
    \frac{2(\lambda t)^{2n+1}e^{-\lambda t}}{(1-e^{-2\lambda
    t})(2n+1)!}dxdy\\
    \nonumber &= \frac{\lambda}{2\pi c}\frac{1}{\sinh(\lambda
    t)}\sum_{n=0}^{\infty}\left(\frac{\lambda}{c}\right)^{2n}\frac{\left[\sqrt{c^2t^2-(x^2+y^2)}\right]^{2n-1}}{(2n)!}dxdy\\
    \nonumber &=\frac{\lambda}{2\pi c}\frac{1}{\sinh(\lambda t)}
    \frac{\cosh\left(\frac{\lambda}{c}\sqrt{c^2t^2-(x^2+y^2)}\right)}{\sqrt{c^2t^2-(x^2+y^2)}}dxdy,
    \end{align}
    as claimed.
    \end{proof}
    The density \eqref{pover} converges to $+\infty$ near the border of $C_{ct}$ and behaves as the fundamental solution of the planar
    wave equation (see Figure 2).
     \begin{figure}
                    \centering
                    \includegraphics[scale=.32]{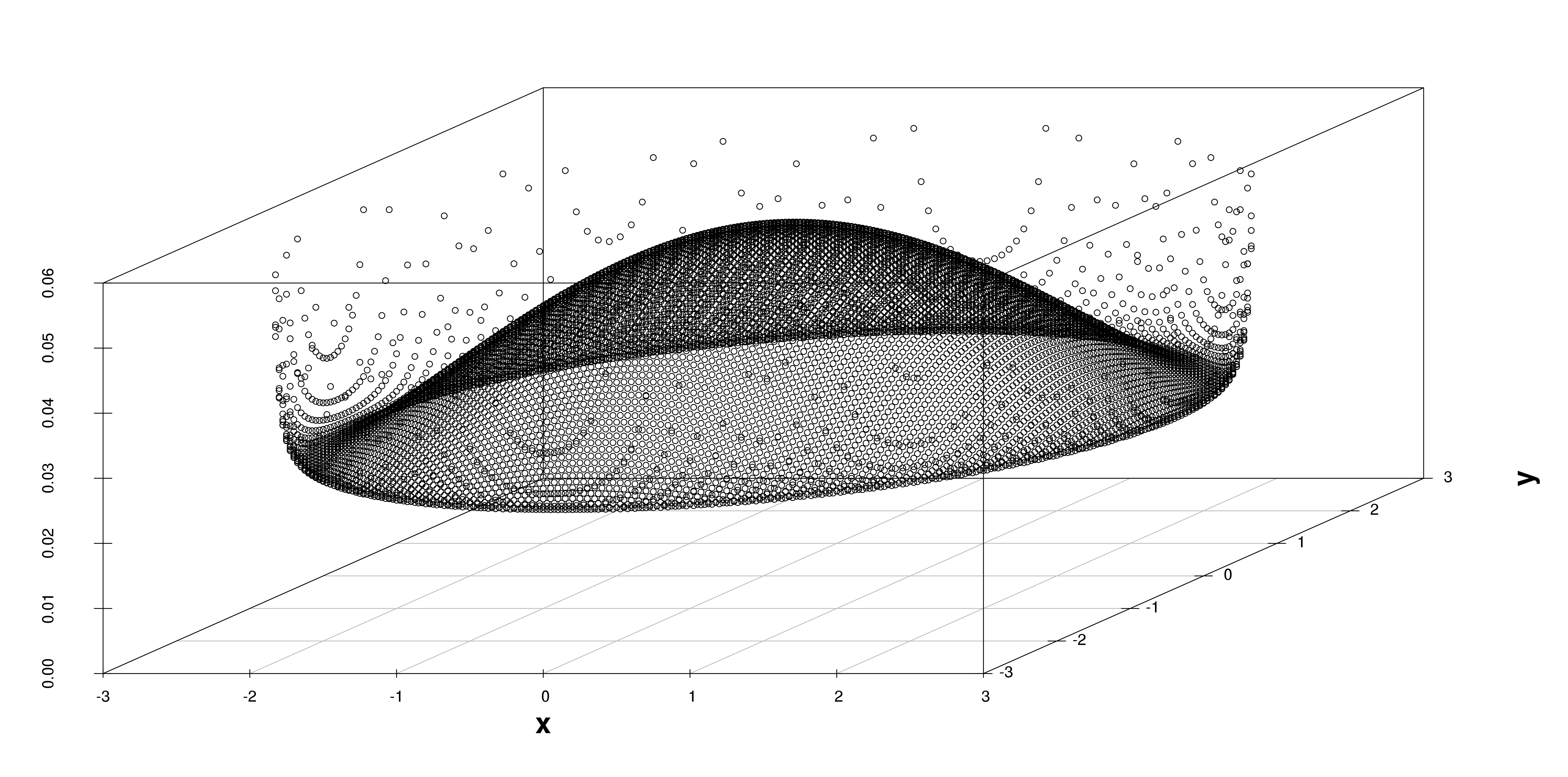}
                    \caption{Behavior of the distribution \eqref{pover} 
                    for $\lambda = c= 1$ and $t= 3$. }
            \end{figure}
    
    We observe that in this case, the planar motion develops
    completely inside the circle $C_{ct}$. Moreover a remarkable
    relation can be found with the process studied in
    \cite{SPA}. 
    We also recall that in \cite{SPA} the governing equation of this
    planar motion was given and it corresponds to the following telegraph
    equation with time-dependent coefficient
    \begin{equation}
    \left(\frac{\partial^2}{\partial t^2}+2\lambda \coth(\lambda t)\frac{\partial}{\partial
    t}-c^2\Delta\right)p(x,y,t)=0.
    \end{equation}
    We now consider the complementary case where the particle changes direction
    only at even-order Poisson times. In this case we have the
    following
    \begin{te}
    The distribution of the planar
    random motion $(X(t),Y(t))$, when the changes of direction are
    taken only at even-order Poisson times is given by
    \begin{equation}\label{pover1}
    P\{X(t)\in dx, Y(t)\in dy\}=\frac{\lambda}{2\pi c}\frac{1}{\cosh(\lambda t)}
    \frac{\sinh\left(\frac{\lambda}{c}\sqrt{c^2t^2-(x^2+y^2)}\right)}{\sqrt{c^2t^2-(x^2+y^2)}}dxdy,
    \end{equation}
    where $(x,y)\in int C_{ct}$.
    \end{te}

    \begin{proof}
    In view of Theorem \ref{kole} we have
    \begin{align}
    \nonumber &P\{X(t)\in dx, Y(t) \in dy\} \\
    \nonumber &=\sum_{n=0}^{\infty}P\{X(t)\in dx, Y(t)\in dy|N(t) = 2n\}
    P\{N(t)= 2n|\bigcup_{k=0}^{\infty}\{N(t)= 2k\}\}\\
    \nonumber &= \sum_{n=0}^{\infty}\frac{(2n)}{2\pi
    (ct)^{2n}}\left[\sqrt{c^2t^2-(x^2+y^2)}\right]^{2n-2}
    \frac{2(\lambda t)^{2n}e^{-\lambda t}}{(1+e^{-2\lambda
    t})(2n)!}dxdy\\
    \nonumber &= \frac{\lambda}{2\pi c}\frac{1}{\cosh(\lambda
    t)}\sum_{n=0}^{\infty}\left(\frac{\lambda}{c}\right)^{2n+1}\frac{\left[\sqrt{c^2t^2-(x^2+y^2)}\right]^{2n+1}}{(2n+1)!}
    dxdy\\
    \nonumber &=\frac{\lambda}{2\pi c}\frac{1}{\cosh(\lambda t)}
    \frac{\sinh\left(\frac{\lambda}{c}\sqrt{c^2t^2-(x^2+y^2)}\right)}{\sqrt{c^2t^2-(x^2+y^2)}}dxdy,
    \end{align}
    as claimed.
    \end{proof}
    In this case the singular component of the distribution is concentrated on $\partial C_{ct}$ and
    is given by
    \begin{equation}
    \int_{-ct}^{+ct}P\{X(t)\in dx, Y(t) \in dy\} = 1-\frac{1}{\cosh(\lambda
    t)},
    \end{equation}
    where
    \begin{equation}\label{bombae}
    \frac{1}{\cosh (\lambda t)}=P\{N(t)= 0|\bigcup_{k=0}^{\infty}\{N(t)=
    2k\}\}.
    \end{equation}
    The density \eqref{pover1} has a bell--shaped form (see Figure 3) with a singular part of weight \eqref{bombae} on  $\partial C_{ct}$.
    \begin{figure}
                        \centering
                        \includegraphics[scale=.32]{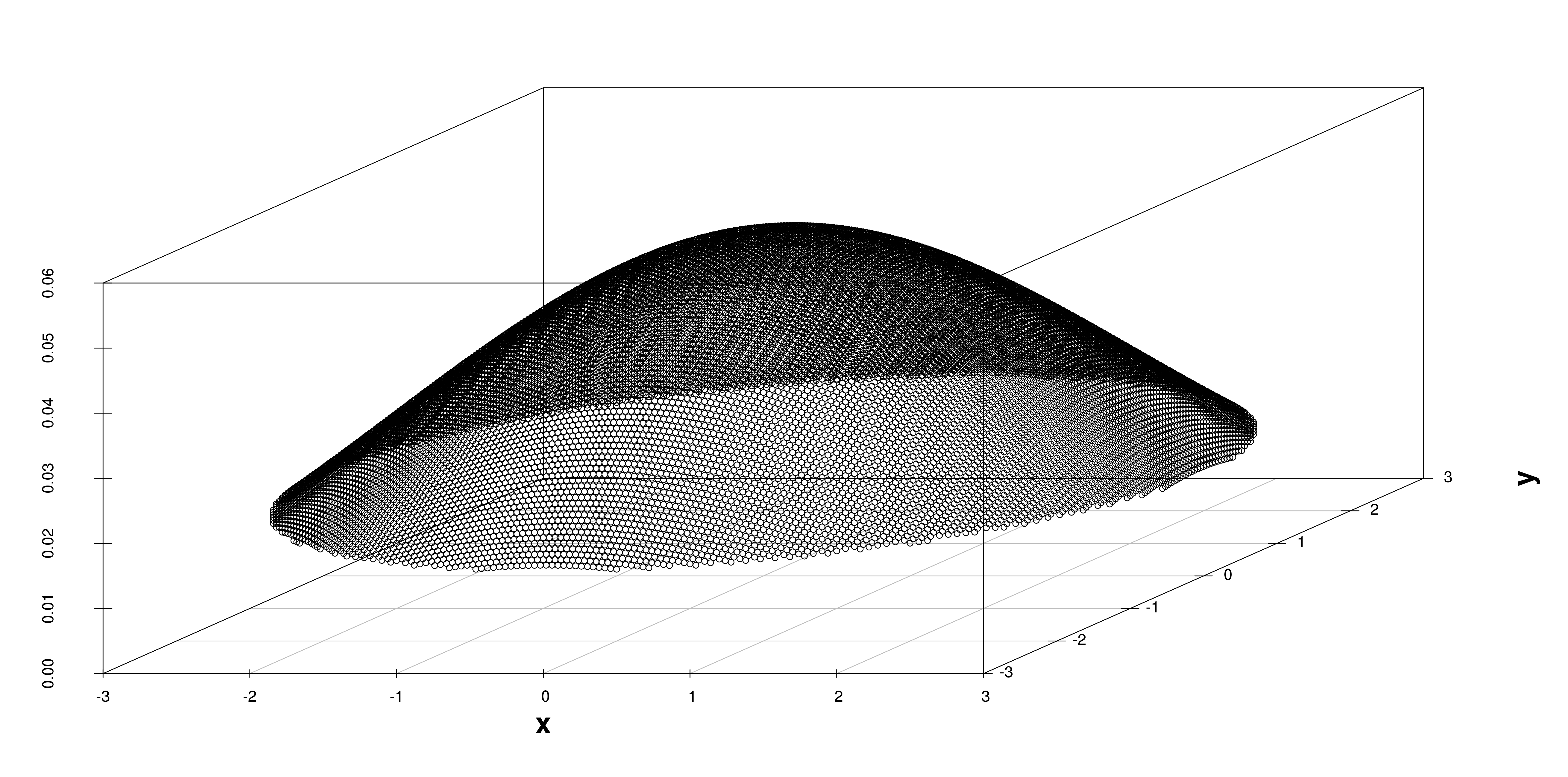}
                        \caption{Behavior of the distribution \eqref{pover1} 
                        for $\lambda = c= 1$ and $t= 3$. }
                \end{figure}

    It is simple to prove the following
    \begin{prop}
    The density law of the planar random motion with density \eqref{pover1} is governed by the
    following telegraph-type equation
    \begin{equation}
     \left(\frac{\partial^2}{\partial t^2}+2\lambda \tanh(\lambda t)\frac{\partial}{\partial
    t}-c^2\Delta\right)p(x,y,t)=0.
    \end{equation}
    \end{prop}

\smallskip

        \textbf{Acknowledgement.} The authors gratefully acknowledge the appreciation of their work 
        by the referee as well as his/her suggestions and comments.

\end{document}